\documentclass[12pt]{amsart}
\usepackage{amssymb,amsthm,amsmath,a4}
\newtheorem{thm}{Theorem}[section]
\newtheorem{theorem}[thm]{Theorem}
\newtheorem{corollary}[thm]{Corollary}
\newtheorem{lemma}[thm]{Lemma}

\theoremstyle{definition}

\theoremstyle{remark}
\newtheorem{remark}[thm]{Remark}
\newtheorem*{acknowledgements}{Acknowledgements}
\numberwithin{equation}{section}

\newcommand{\ab}{\mathbf{a}}

\newcommand{\bb}{\mathbf{b}}
\newcommand{\BR}{\mathrm{B}}
\newcommand{\BC}{\mathcal{B}}

\newcommand{\dr}{\mathrm{d}}
\newcommand{\de}{\mathrm{\delta}}
\newcommand{\DR}{\mathrm{D}}
\newcommand{\DC}{\mathcal{D}}

\newcommand{\eps}{\varepsilon}

\newcommand{\Hf}{\mathfrak{H}}
\newcommand{\HC}{\mathcal{H}}

\newcommand{\la}{\lambda}
\newcommand{\La}{\Lambda}
\newcommand{\LC}{\mathcal{L}}

\newcommand{\NR}{\mathrm{N}}

\newcommand{\Om}{\Omega}

\newcommand{\si}{\sigma}

\newcommand{\C}{\mathbb{C}}

\newcommand{\R}{\mathbb{R}}

\newcommand{\ess}{\mathrm{ess}}

\def\leq {\leqslant}
\def\geq {\geqslant}

\begin{document}

\title[Dirichlet and Neumann counting functions]
{On the comparison of the Dirichlet and Neumann counting functions}

\author{Y. Safarov}
\address{Department of Mathematics, King's College London,
Strand, London, UK} \email{yuri.safarov@kcl.ac.uk}

\date{February 2008}

\dedicatory{To Mikhail Shl\"emovich Birman on his 80-th birthday}

\thanks{The research was supported by the EPSRC grant GR/T25552/01.}


\date{August 2008}

\maketitle

\section*{Introduction}\label{I}

Let $N_\NR(\la)$ and $N_\DR(\la)$ be the counting functions of the
Dirichlet and Neumann Laplacian on a domain $\Om\subset\R^n$. If
$\,\la\,$ is not a Dirichlet or Neumann eigenvalue then
\begin{equation*}
N_\NR(\la)=N_\DR(\la)+g^-(\la)\,,\tag{*}
\end{equation*}
where $g^-(\la)$ denotes the number of negative eigenvalues of the
Dirichlet-to-Neumann map at $\la\in\R$. The equality (*) was proved
in \cite{Fr1} for domains with sufficiently smooth boundaries. L.
Friedlander also noticed that (*) immediately implies Payne's
conjecture for the Laplacian on a bounded domain, according to which
the $\,(k+1)$th Neumann eigenvalue does not exceed the $\,k$th
Dirichlet eigenvalue. Later R. Mazzeo remarked that (*) remains
valid for domains with smooth boundaries in any Riemannian symmetric
space of noncompact type and gave a geometric explanation of
Friedlander's result \cite{M}.

For irregular boundaries, the Dirichlet-to-Neumann map may not be
well-defined and then (*) does not make sense. In 2004 N. Filonov
suggested another proof of Payne's conjecture for the Laplacian
\cite{Fi}. This proof does not use (*) and works for nonsmooth
boundaries. The author assumed that the resolvent of the Neumann
Laplacian on $\,\Om\,$ is compact but this condition can be removed
(see Remark \ref{filonov1}).

The aim of this note is to show that (*) holds for abstract
operators in a Hilbert space $H$, provided that the
Dirichlet-to-Neumann map is understood in a proper sense.
Traditionally, one assumes that the Dirichlet-to-Neumann map is a
family of operators acting in the same space and depending on the
spectral parameter $\,\la\,$ (see Subsection \ref{M3}). In our
understanding, it is a family of operators $\,\BC_\la\,$ generated
by the restrictions of the same sesquilinear form to different
subspaces $\,G_\la\subset H^1\,$. The identity (*) is proved with
the use of special isomorphisms between the subspaces $G_\la$ with
different values of $\la\,$.

This approach is close in spirit to Birman's paper \cite{B1} on
self-adjoint extensions of symmetric operators. In particular, it
removes technical problems related to nonsmooth boundaries and
allows one to extend Payne's conjecture to all operators generated
by differential quadratic forms with constant coefficients on an
arbitrary domain $\Om\subset\R^n$ with $n\ge2$ (see
Corollary~\ref{c-diff}). Another advantage of our scheme is that,
unlike the classical Dirichlet-to-Neumann map, the operators
$\,\BC_\la\,$ do not blow up as $\,\la\,$ passes through isolated
eigenvalues. This enables one to perform more detailed analysis of
the relation between their properties and spectral characteristics
of the Dirichlet and Neumann problems.

The paper is constructed as follows. In Section 1 we introduce some
necessary notation and state the main results. Note that the
notation is deliberately chosen as if $A$ is a second order elliptic
differential operator acting in the Sobolev spaces on a domain,
subject to the Dirichlet or Neumann boundary condition (even though
$H$ does not have to be a function space and the ellipticity is
irrelevant). In Section 2 we prove some simple auxiliary lemmas on
abstract self-adjoint operators. Section 3 is devoted to the proof
of main statements. Finally, Section 4 contains some remarks and
by-product results, which are not needed in our proofs but may be of
interest in themselves.

\begin{acknowledgements}
Preliminary results were reported in the workshop "Spectral analysis
of difference and differential operators" (Banach Center, Warsaw,
2005) organised by J. Janas and S. Naboko. I am grateful to the
organisers and participants for their encouragement and fruitful
discussions. I am also indebted to A. Pushnitski, M. Solomyak and,
especially, to N. Filonov for their useful comments.
\end{acknowledgements}

\section{Basic notation and main results}\label{M}

\subsection{Notation}\label{M1}
We shall always be assuming that $\,\la,\mu\in\R\,$ and
$\,z\in\C\,$.

Let $H$ be an infinite dimensional separable complex Hilbert space.
As usual, $(\cdot,\cdot)$ and $\|\cdot\|$ are the inner product and
norm in $H$, and $\,\dotplus\,$ denotes a direct sum in $\,H\,$. Let
\begin{enumerate}
\item[$\bullet$]$H^1$ be a dense subspace of $H$;
\item[$\bullet$] $\ab[\cdot]$ be a closed positive
    quadratic form on $H^1$ and $\ab[\cdot,\cdot]$ be the
    corresponding sesquilinear form;
\item[$\bullet$] $\,A_\NR\,$ be the self-adjoint operator
    in $H$ generated by the form $\ab[\cdot]$.
\end{enumerate}
We shall consider $H^1$ as a Hilbert space provided with the
inner product $\ab[\cdot,\cdot]$. Let
\begin{enumerate}
\item[$\bullet$] $H_0^1$ be a closed subspace of $H^1$ which
    is dense in $H$;
\item[$\bullet$] $\,A_\DR\,$ be the self-adjoint operator
    in $H$ generated by the restriction of $\ab[\cdot]$ to
    $H_0^1$.
\end{enumerate}

Further on we shall write $\,\BR\,$ instead of $\NR$ or $\DR$
in the case where the corresponding statement holds or
definition refers to the both operators $A_\NR$ and $A_\DR$. In
particular, we shall be using the following notation.
\begin{enumerate}
\item[$\bullet$] $\si(A_\BR)$ and $\si_\ess(A_\BR)$ denote
    the spectrum and the essential spectrum of $A_\BR$.
\item[$\bullet$] $\la_{B,\infty}:=\inf\si_\ess(A_\BR)$.
\item[$\bullet$]
    $\la_{B,1}\leq\la_{B,2}\leq\la_{B,3}\ldots$ are the
    eigenvalues of the operator $A_\BR$ lying in the
    interval $(-\infty,\la_{B,\infty})$ and counted with
    their multiplicities.
    \item[$\bullet$] $\chi_\La$ denotes the characteristic
        function of the Borel set $\La\subset\R$, so that
\item[$\bullet$] $\chi_\La(A_\BR)$ is the spectral
    projection of $A_\BR$ corresponding to $\La$.
\item[$\bullet$]
$E_\BR(z)$ be the orthogonal projection onto $\ker(A_\BR-zI)$ and
$E'_\BR(z):=I-E_\BR(z)$.
\item[$\bullet$]
    $\,N_\BR(\la):=\dim\chi_{(-\infty,\la)}(A_\BR)\,H\,$ is
    the left continuous counting function of the operator
    $A_\BR\,$.
\end{enumerate}
The Rayleigh--Ritz variational formula implies that $N_\DR(\la)\leq
N_\NR(\la)$ or, in other words, $0<\la_{\NR,j}\leq\la_{\DR,j}$ for
all $j=1,2,\ldots,\infty$. We have
$N_\BR(\la)=\#\{j:\la_{\BR,j}<\la\}$ whenever
$\la\leq\la_{\BR,\infty}$ and $N_\BR(\la)=\infty\,$ otherwise.

Let
\begin{enumerate}
\item[$\bullet$] $H_A^1$ be the set of vectors $u\in H^1$
    such that the functionals $v\to \ab[u,v]$ on $H_0^1$
    are $H$-continuous;
\item[$\bullet$] $A$ be the operator acting from $H_A^1$ to $H$ such
    that $(Au,v)=\ab[u,v]$ for all $v\in H_0^1$;
\item[$\bullet$] $G_z:=\{u\in H_A^1\,:\,Au=zu\}\,$ where
    $\,z\in\C$;
\item[$\bullet$] $\bb[u,v]:=\ab[u,v]-(Au,v)\,$ and
    $\,\bb[u]:=\bb[u,u]\,$ where $\,u\in H_A^1\,$ and $\,v\in H^1\,$.
\end{enumerate}
Since the operator $A$ is $H^1$-closed, $G_z$ are closed subspaces
of $H^1$. Denote
\begin{enumerate}
\item[$\bullet$]
$\BC_\la:=\left.(I-\la\Pi'_\la A_\NR^{-1})\right|_{G_\la}\,$ where
\item[$\bullet$]
$\Pi'_\la$ is the $H^1$-orthogonal projection onto $G_\la\,$ (an
explicit formula for $\Pi'_\la$ is given in Subsection~\ref{A3}).
\end{enumerate}
We shall consider $\,\BC_\la\,$ as an operator in $\,G_\la\,$.
Obviously,
\begin{equation}\label{b}
\ab[\BC_\la u,v]\ =\ \ab[u,v]-\la\,(u,v)\ =\ \bb[u,v]\,, \qquad
\forall u,v\in G_\la\,.
\end{equation}
Therefore $\BC_\la$ is a bounded self-adjoint operator in the
Hilbert space $G_\la$ provided with the inner product
$\ab[\cdot,\cdot]\,$.

Let
\begin{enumerate}
\item[$\bullet$] $\si(\BC_\la)\,$ and $\,\si_\ess(\BC_\la)\,$ be the spectrum and
essential spectrum of $\,\BC_\la\,$;
\item[$\bullet$]
$G_\la^0:=\ker\BC_\la\,$,
$\,G_\la^-:=\chi_{(-\infty,0)}(\BC_\la)G_\la\,$ and
$\,G_\la^+:=\chi_{(0,+\infty)}(\BC_\la)G_\la\,$,
\end{enumerate}
where $\,\chi_{(-\infty,0)}(\BC_\la)\,$ and
$\,\chi_{(0,+\infty)}(\BC_\la)\,$ are the corresponding spectral
projections of the operator $\,\BC_\la\,$.

Finally, let
\begin{enumerate}
\item[$\bullet$] $\HC_0\,$ be the subspace of $\,H\,$ spanned by all
    common eigenvectors of $\,A_\NR\,$ and $\,A_\DR\,$;
\item[$\bullet$] $\HC\,$  be the $H$-orthogonal complement
    of $\HC_0$;
\item[$\bullet$]
$\,n_{\NR,\DR}(\la):=\dim E_\BR(\la)\HC_0\,$ and $n_\BR(\la):=\dim
E_\BR(\la)\HC=\dim E_\BR(\la)H-n_{\NR,\DR}(\la)\,$.
\end{enumerate}
Clearly, $\,\HC\,$ and $\,\HC_0\,$ are invariant subspaces of the
operators $\,A_\NR\,$ and $\,A_\DR\,$, whose intersections with
$\,H^1\,$ are $\,H^1$-orthogonal. Similarly, $\,G_\la\bigcap \HC\,$
and $\,G_\la\bigcap\HC_0\,$ are invariant subspaces of
$\,\BC_\la\,$. We have
$\,G_\la\bigcap\HC_0=G_\la^0\bigcap\HC_0=E_\BR(\la)\HC_0\,$ and
$\,\left.\BC_\la\right|_{G_\la\bigcap\HC_0}=0\,$. In particular,
$\,G_\la\bigcap\HC_0=\{0\}\,$ whenever $\,\la\,$ is not an
eigenvalue corresponding to a common eigenvector of the operators
$\,A_\NR\,$ and $\,A_\DR\,$.

\subsection{Main results}\label{M2}
The following lemma implies that the restriction
$\,\left.\BC_\la\right|_\HC\,$ analytically depends on $\,\la\,$
outside the intersection of the essential spectra
$\,\si_\ess\left(A_\NR\right)\,$ and
$\,\si_\ess\left(A_\DR\right)\,$.

\begin{lemma}\label{l1-projections}
The $H^1$-orthogonal projection onto $\,G_\la\bigcap\HC\,$ is an
analytic operator-valued function of $\,\la\,$ on the set
$\,\R\setminus\left(\si_\ess\left(A_\NR\right)
\bigcap\si_\ess\left(A_\DR\right)\right)\,$.
\end{lemma}

One can easily show that
\begin{equation}\label{eq-g0}
E_\NR(\la)H+E_\DR(\la)H\ \subset\ G_\la^0\,,\qquad\forall\la\in\R\,,
\end{equation}
(see Subsection \ref{Proofs3}). The next lemma is less obvious.

\begin{lemma}\label{l-g0}
If $\,\la\not\in\si_\ess(A_\NR)\bigcap\si_\ess(A_\DR)\,$ then we
have $\,G_\la^0=E_\DR(\la)H+E_\NR(\la)H\,$. If
$\,\la\not\in\si_\ess(A_\NR)\bigcup\si_\ess(A_\DR)\,$ then the point
$\,0\,$ does not belong to the essential spectrum of the operator
$\,\BC_\la\,$.
\end{lemma}

Lemmas~\ref{l1-projections} and \ref{l-g0} imply

\begin{theorem}\label{t1-main}
Let $\,\la\not\not\in\si_\ess(A_\NR)\bigcup\si_\ess(A_\DR)\,$. Then
for each sufficiently small $\,\eps>0\,$ there exists $\,\de>0\,$
such that the intersection $\,(-\eps,\eps)\bigcap\si(\BC_\mu)\,$
consists of
\begin{enumerate}
\item[(1)]
$\,n_\NR(\mu)+n_\DR(\mu)+n_{\NR,\DR}(\mu)\,$ zero eigenvalues if
$\,\mu=\la\,$,
\item[(2)]
$\,n_\DR(\la)\,$ negative  and $\,n_\NR(\la)\,$ positive eigenvalues
if $\,\mu\in(\la-\de,\la)\,$,
\item[(3)]
$\,n_\NR(\la)\,$ negative  and $\,n_\DR(\la)\,$ positive eigenvalues
if $\,\mu\in(\la,\la+\de)\,$
\end{enumerate}
(as usual, the eigenvalues are counted according to their
multiplicities).
\end{theorem}

\begin{remark}\label{r-main}
By Lemma~\ref{l-g0}, if
$\,\la\not\in\si_\ess(A_\NR)\bigcup\si_\ess(A_\DR)\,$ then
$\,[-\eps,\eps]\bigcap\si_\ess\left(\BC_\mu\right)=\varnothing\,$
for all sufficiently small $\,\eps,\de>0\,$ and all
$\,\mu\in[\la-\de,\la+\de]\,$. By Lemma~\ref{l1-projections}, the
eigenvalues $\,\nu_j(\mu)\,$ of the restrictions
$\,\left.\BC_\mu\right|_{G_\mu\bigcap\HC}\,$ lying in
$\,(-\eps,\eps)\,$ are continuous function of
$\,\mu\in(\la-\de,\la+\de)\,$. Therefore, if $\,\eps\,$ and
$\,\de\,$ are small enough then $\,\nu_j(\mu)\in(-\eps,\eps)\,$ for
some $\,\mu\in(\la-\de,\la+\de)\,$ if and only if
$\,\nu_j(\la)=0\,$. Theorem~\ref{t1-main} states that
$\,n_\DR(\mu)\,$ eigenvalues $\,\nu_j(\mu)\,$ change their sign from
minus to plus and $\,n_\NR(\mu)\,$ eigenvalues $\,\nu_j(\mu)\,$
change their sign from plus to minus as $\,\mu\,$ passes through the
eigenvalue $\,\la\,$. At the point $\,\la\,$ all these eigenvalues
are equal to zero and, in addition, there are $\,n_{\NR,\DR}(\la)\,$
zero eigenvalues of the restriction
$\,\left.\BC_\la\right|_{G_\la\bigcap\HC_0}\,$.
\end{remark}

\begin{remark}\label{friedlender0}
A similar result was obtained in \cite{Fr1} and \cite{M} for
differential operators on domains with smooth boundaries under
the additional assumption that their spectra are discrete.
Theorem~\ref{t1-main} holds in the abstract setting and remains
valid for $\,\la\,$ lying in the gaps of the essential spectra.
\end{remark}

\begin{corollary}\label{c-main}
Let $\,a<b\,$. If $\,[a,b]\bigcap\si_\ess(A_\NR)=\varnothing\,$ and
$\,[a,b]\bigcap\si_\ess(A_\DR)=\varnothing\,$ then
\begin{equation}\label{eq-main}
\dim G_b^-\ =\ \dim G_a^-\,+\,\dim\chi_{[a,b)}(A_\NR)\HC
\,-\,\dim\chi_{(a,b]}(A_\DR)\HC\,.
\end{equation}
\end{corollary}

If $\,a<\inf \si(A_\NR)\,$ then $\,G_a^-=\{0\}\,$ and, according to
the next theorem, the equality \eqref{eq-main} remains valid for
$\,b\in[\la_{\NR,\infty},\la_{\DR,\infty})\,$.

\begin{theorem}\label{t2-main}
$\;N_\NR(\la)=N_\DR(\la)+n_\DR(\la)+\dim G_\la^-\;$ for all
$\,\la<\la_{\DR,\infty}\,$.
\end{theorem}

\begin{remark}\label{birman1}
By Theorem~\ref{t2-main}, $\,N_\NR(\la)=\dim G_\la^-\;$ for all
$\la$ lying below $\,\si(A_\DR)\,$. In the case where $\,A_\NR\,$
and $\,A_\DR\,$ are self-adjoint extensions of the same symmetric
operator defined on $\DC(A_\NR)\bigcap\DC(A_\DR)$, the above
identity was obtained by M.S. Birman \cite{B1} (see also \cite{B2}).
Theorem~\ref{t2-main} extends Birman's result to all
$\la<\la_{\DR,\infty}$ in a slightly more general setting (see
Subsection~\ref{R1}).
\end{remark}

\begin{remark}\label{filonov1}
N. Filonov noticed in \cite{Fi} that, for the Laplacian on an
arbitrary domain $\,\Omega\subset\R^n\,$,
\begin{equation}\label{filonov}
\ab[u]\ \leq\ \la\,\|u\|^2\,,\qquad \forall u\in
\chi_{[0,\la]}(A_\DR)H+E_\NR(\la)H+G^0_\la+G_\la^-\,.
\end{equation}
Similar arguments show that \eqref{filonov} holds for any pair of
abstract operators $A_\DR$ and $A_\NR$ (see
Subsection~\ref{Proofs2}). The estimate \eqref{filonov} immediately
implies that
\begin{equation}\label{upper}
N_\NR(\la)\ \geq\ N_\DR(\la)\,+\,n_\DR(\la)\,+\,\dim
G_\la^-\,,\qquad\forall\la\in\R\,.
\end{equation}
The inequality \eqref{upper} is sufficient to prove Payne's
conjecture for the Laplacian on a bounded domain (see the proof of
Corollary~\ref{c-diff}).
\end{remark}

\begin{remark}\label{filonov2}
The equality $\,N_\NR(\la)=N_\DR(\la)+n_\DR(\la)+\dim G_\la^-\,$
remains valid for all $\,\la>\la_{\DR,\infty}\,$ because
$\,N_\NR(\la)=N_\DR(\la)=\infty$. However, as was pointed out by N.
Filonov, it may not be true for $\,\la=\la_{\DR,\infty}\,$.
\end{remark}

\begin{remark}\label{r-g0}
Let $\,\la=\la_{\DR,k}<\la_{\NR,\infty}\,$.
Theorem~\ref{t2-main} implies that the number of eigenvalues
$\,\la_{\NR,j}\,$ lying below $\,\la_{\DR,k}\,$ is equal to
$\,k-1+n_\DR(\la_{\DR,k})+\dim G_{\la_{\DR,k}}^-\,$. Therefore
\begin{enumerate}
\item[(1)] $\,\la_{\NR,k+q_k+p_k-1}<\la_{\DR,k}\,$, where
    $\,p_k:=\dim G_{\la_{\DR,k}}^-\,$ and
    $\,q_k:=n_\DR(\la_{\DR,k})\,$.
\end{enumerate}
If $\,n_\DR(\la_{\DR,k})=0\,$ then
$\,\la_{\NR,k+q_k+p_k}=\la_{\NR,k+p_k}=\la_{\DR,k}\,$; if
$\,n_\DR(\la_{\DR,k})\ne0\,$ then $\,q_k\geq1\,$. Thus we
always have
\begin{enumerate}
\item[(2)] $\,\la_{\NR,k+p_k}\leq\la_{\DR,k}\,$.
\end{enumerate}
Note that the estimates (1) and (2) are actually consequences
of \eqref{upper}. These estimates and Lemma~\ref{l-g0} imply
that
\begin{enumerate}
\item[(3)] $\,\la_{\NR,k+1}\leq\la_{\DR,k}\,$ whenever
    there exists a vector $\,u\in G_{\la_{\DR,k}}\,$, such
    that $\bb[u]\leq0$ and $u\not\in\DC(A_\DR)\,$;
\item[(4)] $\,\la_{\NR,k+1}<\la_{\DR,k}\,$ whenever
    $\,n_\DR(\la_{\DR,k})\geq1\,$ and there exist two
    vectors $\,u_1,u_2\in G_{\la_{\DR,k}}\,$, such that
    $\bb[u_1]\leq0$, $\bb[u_2]\leq0$ and the linear
    subspace spanned by $\,u_1$ and $u_2\,$ does not
    contain Neumann eigenvectors.
\end{enumerate}
Indeed, if $G_{\la_{\DR,k}}^-\geq1$ then (3) and (4) follow from (2)
and (1) respectively. If $G_{\la_{\DR,k}}^-=\{0\}$ then $\,u\in
G_{\la_{\DR,k}}^0\,$ and $\,u_1,u_2\in G_{\la_{\DR,k}}^0\,$. The
inclusion $\,u\in G_{\la_{\DR,k}}^0\,$ implies that
$\,\la_{\DR,k}\,$ is also a Neumann eigenvalue and, consequently,
$\,\la_{\NR,k+1}=\la_{\DR,k}\,$. The inclusions $\,u_1,u_2\in
G_{\la_{\DR,k}}^0\,$ imply that $\,n_\DR(\la_{\DR,k})\geq2\,$
(otherwise a linear combination of $u_1$ and $u_2$ would belong to
$\,E_\NR(\la_{\DR,k})H\,$).
\end{remark}

Lemma~\ref{l-g0} and Theorem~\ref{t2-main} also imply

\begin{corollary}\label{c-resolvents}
If $\,\la<\la_{\DR,\infty}\,$ and
$\,\la\not\in\si(A_\NR)\bigcup\si(A_\DR)\,$ then the number of
negative eigenvalues of the self-adjoint operator
$\,R'(\la):=(A_\NR-\la I)^{-1}-(A_\DR-\la I)^{-1}\,$ in $\,H\,$
coincides with $\,N_\NR(\la)-N_\DR(\la)\,$.
\end{corollary}

Obviously, the number of negative eigenvalues of the operator
$\,(A_\BR-\la I)^{-1}\,$ jumps by
$\,n_\BR(\la_0)+n_{\NR,\DR}(\la_0)\,$ as $\,\la\,$ passes
through an eigenvalue $\,\la_0\,$. Corollary~\ref{c-resolvents}
shows that the corresponding jump for $\,R'(\la)\,$ is equal to
$\,n_\NR(\la_0)-n_\DR(\la_0)\,$, as if $\,R'(\la)\,$ were the
orthogonal sum of the operators $\,(A_\NR-\la I)^{-1}\,$ and
$\,-(A_\DR-\la I)^{-1}\,$.

\subsection{The Dirichlet-to-Neumann map}\label{M3}
In the theory of boundary value problems, it is often possible to
construct a linear isomorphism $W:G_\la\to\Hf\,$, where $\Hf$ is a
Hilbert space of functions defined on the boundary. Then one can
consider the operator $\,W\BC_\la W^{-1}:\Hf\to\Hf\,$ instead of
$\BC_\la$. Clearly, these two operators have the same eigenvalues.
If $H^1$, $H_0^1$ are the Sobolev spaces and $Wv$ is the restriction
of $v$ to the boundary then $W\BC_\la W^{-1}$ is usually called the
Dirichlet-to-Neumann map. This scheme works under certain smoothness
conditions on the boundary and the coefficients, whereas our
approach does not rely on the existence of an auxiliary operator $W$
and does not require any additional assumptions.

\subsection{Applications to boundary value problems}\label{M4}
Let $\,\Om\,$ be an arbitrary open subset of $\R^n\,$ with
$\,n\geq2\,$. Consider a differential operator $\,L\,$ acting from
the space of $\,m$-vector functions $\,C^\infty(\Om,\C^m)\,$ into
the space of $\,l$-vector functions $\,C^\infty(\Om,\C^l)\,$ and
denote by $\,L^*\,$ its formal adjoint. Let us assume that the form
$\,\int_\Om|Lu(x)|^2\,\dr x\,$ with domain
$\,C^\infty(\Om,\C^m)\bigcap L_2(\Om,\C^m)\,$ is strictly positive
and closable in $\,H=L_2(\Om,\C^m)\,$, and denote its closure by
$\,\ab[u]\,$. If $\,H^1:=\DC(\ab)\,$ and $\,H^1_0\,$ is the
$\,H^1$-closure of $\,C_0^\infty(\Om)\,$ then $\,A=L^*L\,$ and
$\,A_\BR\,$ is the differential operator $\,A\,$ with the
corresponding boundary condition.

\begin{corollary}\label{c-diff}
Let $L$ be an operator with constant coefficients. Then
$\,\la_{\NR,k+1}\leq\la_{\DR,k}\,$ for all eigenvalues
$\,\la_{\DR,k}\in(0,\la_{\NR,\infty})\,$. If at least one Dirichlet
eigenfunction corresponding to $\,\la_{\DR,k}\,$ does not satisfy
the Neumann boundary condition then $\,\la_{\NR,k+1}<\la_{\DR,k}\,$.
\end{corollary}

\begin{remark}\label{friedlender1}
Our proof of Corollary~\ref{c-diff} uses the exponential functions
$\,u_\xi(x)=e^{ix\cdot\xi}\,$ and is very similar to the proof of
the Payne conjecture given in \cite{Fr1}. The main difference is
that L. Friendlender considered the Dirichlet-to-Neumann map and
therefore had to assume that the boundary is smooth enough.
\end{remark}

\begin{remark}\label{friedlender2}
If $A$ is the Laplacian on a convex $\,n$-dimensional domain with
sufficiently smooth boundary then $\,\la_{\NR,k+n}<\la_{\DR,k}\,$.
This estimate was obtained in \cite{LW}. Later L. Friedlander found
another proof, based on the fact that $\,G_\la^0\bigcup G_\la^-\,$
contains all first order derivatives $D_ju$ of the Dirichlet
eigenfunctions $u\in E_\DR(\la)H$ (the derivatives obviously belong
to $G_\la$, and the estimate $\bb[D_ju]\leq0$ is a consequence of
the convexity). The inclusion $D_ju\in G_\la^0\bigcup G_\la^-\,$
also implies that $\,N_\NR(\la)\geq N_\DR(\la)+2n_\DR(\la)\,$ (see
\cite{Fr2} for details).
\end{remark}

\section{Further notation and auxiliary results}\label{A}

\subsection{}\label{A1}
The inverse $A_\NR^{-1}$ is a bounded self-adjoint operator in
$H^1\,$ because $\,\ab[A_\NR^{-1}u,v]=(u,v)\,$ for all $\,u,v\in
H^1\,$. Since $\,\ab[u,v]=\la\,(u,v)\,$ for all $\,v\in
E_\NR(\la)H\,$ and $\,v\in H^1\,$, its spectral projections
$\,E_\NR(\la)\,$ are $H^1$-orthogonal.

Let
\begin{enumerate}
\item[$\bullet$] $\,\Pi_0\,$ be the orthogonal projection
    in $\,H^1\,$ onto $\,H_0^1\,$.
\end{enumerate}
From the definition of $G_z$ it clear that $\,\Pi'_0=I-\Pi_0\,$
(this well known result can be found, for example, in \cite{K} or
\cite[Chapter 10, Section 3]{BS}). Since
$\ab[A_\DR^{-1}u,v]=(u,v)=\ab[A_\NR^{-1}u,v]$ for all $u\in H^1$ and
$v\in H_0^1$, we have $\,A_\DR^{-1}=\Pi_0A_\NR^{-1}\,$ and
$\DC(A_\DR)=\Pi_0\DC(A_\NR)$. The following simple lemma is also
well known in the theory of self-adjoint extensions.

\begin{lemma}\label{l-a2}
We have $\,H_A^1=G_0\dotplus\DC(A_\BR)\,$. If $\,w_0\in G_0\,$ and
$\,w_\BR\in\DC(A_\BR)\,$ then $\,A(w_0+w_\BR)=A_\BR w_\BR\,$.
\end{lemma}

\begin{proof}
Obviously, $\,G_0\dotplus\DC(A_\BR)\subset H_A^1\,$. On the other
hand, if $v\in H_A^1$ then there exists $\tilde v\in H\,$ such that
$\,(u,\tilde v)=\ab[u,v]$ for all $u\in H_0^1$. Since
$\,(u,v)=\ab[u,A_\NR^{-1}v]\,$, this implies that
$\,\Pi_0v=\Pi_0A_\NR^{-1}\tilde v=A_\DR^{-1}\tilde v\,$. Therefore
$\,v=\Pi'_0v+A_\DR^{-1}\tilde v\,$ and
$\,v=\Pi'_0\left(v-A_\NR^{-1}\tilde v\right)+A_\NR^{-1}\tilde v\,$.
These equalities imply the first statement.

If $w_0\in G_0$ and $w_\BR\in\DC(A_\BR)$ then
$\,\ab[w_0+w_\BR,v]=\ab[w_\BR,v]=(A_\BR w_\BR,v)\,$, for all $\,v\in
H_0^1\,$. This proves the second statement.
\end{proof}

\subsection{}\label{A2} By Lemma~\ref{l-a2}, $H_A^1$ is dense in $H^1$. Since
$$
\ab[A_\DR^{-1}Au,v]=\ab[A_\DR^{-1}Au,\Pi_0v] =(A
A_\DR^{-1}Au,\Pi_0v)=(Au,\Pi_0v)=\ab[\Pi_0u,v]
$$
for all $\,u,v\in H_A^1\,$, this implies that
$A_\DR^{-1}A=\left.\Pi_0\right|_{H_A^1}$. Consequently,
\begin{multline}\label{a1}
G_z\ :=\ \ker(A-zI)\ =\ \ker A_\DR^{-1}(A-zI)\\ =\
\ker(\Pi_0-zA_\DR^{-1})\ =\ \ker\Pi_0(I-zA_\BR^{-1})\,,\qquad\forall
z\in\C\,.
\end{multline}
By \eqref{a1}, we have $(I-zA_\BR^{-1})\,G_z\subset
G_0\bigcap(I-zA_\BR^{-1})H^1$. On the other hand, if
$(I-zA_\BR^{-1})u\in G_0$ then $u\in G_z$ because
$(A-zI)u=A(I-zA_\BR^{-1})u=0$. Therefore
\begin{equation}\label{a2}
(I-zA_\BR^{-1})\,G_z\ =\
G_0\bigcap(I-zA_\BR^{-1})H^1\,,\qquad\forall z\in\C\,.
\end{equation}

Let
\begin{enumerate}
\item[$\bullet$] $R_\BR(z):=(A_\BR-zI)^{-1}\,$ be the resolvent of $A_\BR$.
\end{enumerate}
For each $z\not\in\si_\ess(A_\BR)$, the operator
$\,R_\BR(z)E'_\BR(z)\,$ is bounded  from $H$ to $H^1\,$,
$$
\ker\left(R_\BR(z)E'_\BR(z)\right)=E_\BR(z)H\,,\qquad
R_\BR(z)E'_\BR(z)H\subset E'_\BR(z)\DC(A_\BR)\subset H_A^1
$$
and $\,(A-zI)R_\BR(z)E'_\BR(z)=E'_\BR(z)\,$. We also have
\begin{equation}\label{a3}
\left.(I-zA_\BR^{-1})^{-1}\right|_{E'_\BR(z)H}\ =\
\left(I+zR_\BR(z)\right)E'_\BR(z)\,,\qquad\forall
z\not\in\si_\ess(A_\BR)\,,
\end{equation}
where the operators in the right and left hand sides map
$\,E'_\BR(\la)H^1\,$ onto $\,E'_\BR(\la)H^1\,$ and are
$\,H^1$-bounded. This implies that
$\,(I-zA_\BR^{-1})H^1=E'_\BR(z)H^1\,$ and, in view of \eqref{a2},
\begin{equation}\label{a4}
(I-zA_\BR^{-1})\,G_z\ =\ G_0\bigcap E'_\BR(z)H^1\,, \qquad\forall
z\not\in\si_\ess(A_\BR)\,.
\end{equation}

\subsection{}\label{A3}
Denote $\,T_z:=\left.(I-zA_\NR^{-1})\right|_{H_0^1\bigcap\HC}\,$ and
$\,T_z^\star:=\left.\Pi_0(I-zA_\NR^{-1})\right|_\HC\,$. Let
$\,\Sigma\,$ be the set of points $\,z\in\C$ such that the spectrum
of the operator $\,T_z^\star\,T_z:H_0^1\bigcap\HC\to
H_0^1\bigcap\HC\,$ contains the point $\,0\,$, and let
\begin{enumerate}
\item[$\bullet$]
$\,\Pi(z):=T_z\,(T_z^\star\,T_z)^{-1}\,T_z^\star\,$ and
$\,\Pi'(z):=I-\Pi(z)\,$, where $\,z\in\C\setminus\Sigma\,$.
\end{enumerate}
By \eqref{a1}, we have $\,G_z\bigcap\HC=\ker T_z^\star\,$. Since
$\,T_z^\star\Pi'(z)=0\,$ and $\,\Pi'(z)u=u\,$ for all $\,u\in\ker
T_z^\star\,$, this implies that $\,\Pi'(z)\,$ is a projection onto
$\,G_z\bigcap\HC\,$ in $\,H^1\bigcap\HC\,$. Its $H^1$-adjoint
coincides with $\,\Pi'(\bar z)\,$; in particular, $\,\Pi'(\la)\,$ is
the $\,H^1$-orthogonal projection onto $\,G_z\bigcap\HC\,$ in
$\,\HC\,$. Thus we obtain
\begin{equation}\label{a5}
\Pi'_\la\ =\ \Pi'(\la)\oplus
\left.E_\BR(\la)\right|_{H^1\bigcap\HC_0}\,,\qquad\forall\la\in\R\setminus\Sigma\,.
\end{equation}

\subsection{}\label{A4}
If $z\in\C\setminus\si_\ess(A_\BR)$, let
\begin{enumerate}
\item[$\bullet$]
$P_\BR(z):=R_\BR(z)E'_\BR(z)\,(A-zI)\,$ and
$\,P'_\BR(z):=I-P_\BR(z)\,$;
\item[$\bullet$]
$\,G_{z,B}:=\{v\in H_A^1:(A-zI)v\in E_\BR(z)H\}\,$.
\end{enumerate}
The operators $\,P_\BR(z)\,$ and $\,P'_\BR(z)\,$ are projections in
$\,H_A^1\,$ because
\begin{multline*}
P^2_\BR(z)\ =\
R_\BR(z)E'_\BR(z)\,(A_\BR-zI)R_\BR(z)E'_\BR(z)\,(A-zI)\\ =\
E'_\BR(z)R_\BR(z)E'_\BR(z)\,(A-zI)=P_\BR(z)\,.
\end{multline*}
One
can easily show that $\,P_\BR(z)H_A^1=E'_\BR(z)\DC(A_\BR)\,$ and
$\,P'_\BR(z)H_A^1=G_{z,B}\,$. The subspace $\,G_{z,B}\,$ is the
inverse image of $\,E_\BR(z)H\,$ by the map $\,A-zI\,$, whereas
$\,G_z\,$ is the kernel of $\,A-zI\,$. Therefore $\,G_z\subset
G_{z,B}\,$ and the dimension of the quotient space $\,G_{z,B}/G_z\,$
does not exceed $\,n_\BR(z)+n_{\NR,\DR}(z)\,$. This implies that the
subspaces $\,G_{z,B}\,$ are $\,H^1$-closed for all
$\,z\in\C\setminus\si_\ess(A_\BR)\,$.

If $z\not\in\si(A_\BR)$ then $\,P_\BR(z)H_A^1=\DC(A_\BR)\,$ and
$\,P'_\BR(z)H_A^1=G_z=G_{z,B}\,$. In particular,
$\,\left.P_\DR(0)\right|_{H_A^1}=\left.\Pi_0\right|_{H_A^1}\,$,
$\,\left.P'_\DR(0)\right|_{H_A^1}=\left.\Pi'_0\right|_{H_A^1}\,$ and
$\,H_A^1=P'_\BR(0)H_A^1\dotplus P_\BR(0)H_A^1\,$ is the
decomposition discussed in Lemma~\ref{l-a2}. By direct calculation,
if $\,u,v\in H_A^1$ and $\la,\mu\in\R\,$ then
\begin{multline}\label{a6}
\bb[P'_\NR(\la)u,P'_\NR(\mu)v]\\
=\ \bb[u,v]-\bigl(u,E'_\NR(\mu)(A-\mu I)v\bigr)+\bigl((A-\mu
I)u,R_\NR(\mu)E'_\BR(z)(A-\mu I)v\bigr)
\end{multline}
and
\begin{multline}\label{a7}
\bb[P'_\DR(\la)u,P'_\DR(\mu)v]\\
=\ \bb[u,v]+\bigl(E'_\DR(\la)(A-\la
I)u,v\bigr)-\bigl(R_\DR(\la)E'_\BR(z)(A-\la I)u,(A-\la I)v\bigr).
\end{multline}

\section{Proofs of main results}\label{Proofs}

\subsection{Proof of Lemma~\ref{l1-projections}}\label{Proofs1}

The operator-valued function $T_z^\star\,T_z$ is analytic.
Therefore, for each $\,z_0\not\in\Sigma\,$, the inverse
operator $\,(T_z^\star\,T_z)^{-1}\,$ exists and analytically
depends on $z$ in a sufficiently small neighbourhood of
$\,z_0\,$ (see, for instance, \cite[Section 1.8]{Ya}). Thus it
is sufficient to show that
$\,\Sigma\subset\left(\si_\ess(A_\NR)\bigcap\si_\ess(A_\DR)\right)\,$.

Let us fix $\,\la\in\R\,$, and let $\,u\in\HC\,$. Since
$\,T_\la^\star\,$ is the $\,H^1$-adjoint to $\,T_\la\,$, we have
$T_\la^\star\,T_\la u=0$ if and only if $T_\la u=0$. The latter
means that $u\in\DC(A_\NR)\bigcap\DC(A_\DR)$ and $A_\NR u=A_\DR
u=\la u$. Since $\,u\in\HC\,$, it is only possible if $u=0$. This
implies that $\,\ker(T_\la^\star\,T_\la)=\{0\}\,$.

Assume that the essential spectrum of the operator
$T_\la^\star\,T_\la$ contains the point $\,0\,$. Then, for any given
finite dimensional subspace $\,\LC$, there exists a sequence of
$\,H^1$-orthogonal vectors $\,u_n\in H_0^1\bigcap\HC\,$ such that
$\,u_n\,$ are $\,H$-orthogonal to $\,\LC\,$, $\ab[u_n]=1$ and
$\,\ab[T_\la^\star\,T_\la u_n,u_n]=\ab[T_\la u_n]\to0\,$ as
$n\to\infty$. Clearly, $\,T_\la u_n=u_n-\la\,A_\NR^{-1}u_n\to0\,$
and $\,\Pi_0T_\la u_n=u_n-\la\,A_\DR^{-1}u_n\to0\,$ in $H^1$. If
$\,\la\not\in\si_\ess(A_\BR)\,$ then, by \eqref{a3}, we have
$\,\ab\left[u-\la\,A_\BR^{-1}u\right]\geq C\,\ab[u]\,$ with some
positive constant $\,C\,$ for all vectors $\,u\in H^1\,$ which are
$\,H$-orthogonal to $\,\LC=E_\BR(\la)H\,$. This implies that
$\,\la\in\si_\ess(A_\NR)\bigcap\si_\ess(A_\DR)\,$.

\subsection{Proof of the estimate~\eqref{filonov}}\label{Proofs2}
We have
\begin{eqnarray}
\ab[u,v]\ =\ \la\,(u,v)\,, &\forall u\in E_\NR(\la)H\,, \ \forall
v\in H^1\,,\label{p1}\\
\ab[u,v]\ =\ \la\,(u,v)\,, &\forall u\in \DC(A_\DR)\,,\ \forall v\in
G_\la\,.\label{p2}
\end{eqnarray}
Let $u=u_1+u_2+u_3\,$, where $u_1\in\chi_{[0,\la]}(A_\DR)H$, $u_2\in
E_\NR(\la)$ and $u_3\in G_\la^0+G_\la^-$. Then \eqref{p1} and
\eqref{p2} imply
$$
\ab[u]-\la\,\|u\|^2\ =\
\ab[u_1]-\la\,\|u_1\|^2+\ab[u_3]-\la\,\|u_3\|^2\ =\ \left((A_\DR-\la
I)u_1,u_1\right)+\ab[\BC_\la u_3,u_3]\ \leq\ 0\,.
$$

\subsection{Proof of Lemma~\ref{l-g0}}\label{Proofs3}
The inclusion \eqref{eq-g0} immediately follows from \eqref{p1} and
\eqref{p2}.

Assume that $\,\la\not\in\si_\ess(A_\NR)\bigcap\si_\ess(A_\DR)\,$
and $\,u\in G_\la^0\,$. Then $\,(I-\la A_\NR^{-1})u\,$ is
$\,H^1$-orthogonal to $\,G_\la\,$. In view of \eqref{a5}, this means
that $\,(I-\la A_\NR^{-1})u=\Pi(\la)v\,$ for some $\,v\in
H^1\bigcap\HC\,$. Therefore $\,(I-\la A_\NR^{-1})u=(I-\la
A_\NR^{-1})w\,$ for some $\,w\in H_0^1\,$, which is equivalent to
the inclusion $\,u\in
G_\la\bigcap\left(E_\NR(\la)H+H_0^1\right)=E_\NR(\la)H+E_\DR(\la)H\,$.

Assume now that
$\,\la\not\in\si_\ess(A_\NR)\bigcup\si_\ess(A_\DR)\,$. If the point
$\,0\,$ belongs to the essential spectrum of the operator
$\,\BC_\la\,$ then there exists a sequence of vectors $\,u_n\in
G_\la\ominus G_\la^0\,$ such that $\,\ab[u_n]=1\,$ and
$\,\ab[\BC_\la u_n]\to0\,$ as $\,n\to\infty\,$. Moreover, since
$\,\dim E_\NR(\la)H<\infty\,$, we can choose the sequence
$\,\{u_n\}\,$ in such a way that
$$
u_n-(T_\la^\star\,T_\la)^{-1}\,\Pi_0(I-\la A_\NR^{-1})^2 u_n\ \in\
E'_\NR(\la)\,,\qquad\forall n=1,2,\ldots,
$$
where $\,T_\la\,$ and $\,T_\la^\star\,$ are the operators defines in
Subsection~\ref{A3}. Then, by \eqref{a5},
\begin{multline*}
\BC_\la u_n\ =\ (I-\la A_\NR^{-1})u_n-\Pi_\la(I-\la A_\NR^{-1})u_n\\
=\ \left(I-\la A_\NR^{-1}\right)\left(u_n
-(T_\la^\star\,T_\la)^{-1}\,\Pi_0(I-\la A_\NR^{-1})^2u_n\right)\
\to\ 0
\end{multline*}
in $\,H^1\,$, and \eqref{a3} implies that
\begin{multline*}
\ab\left[u_n -(T_\la^\star\,T_\la)^{-1}\,\Pi_0(I-\la
A_\NR^{-1})^2u_n\right]\\ =\
\ab\left[\Pi'_0u_n\right]+\ab\left[\Pi_0u_n
-(T_\la^\star\,T_\la)^{-1}\,\Pi_0(I-\la A_\NR^{-1})^2u_n\right]\
\to\ 0\,.
\end{multline*}
Therefore $\,\ab[\Pi'_0u_n]=\ab[(I-A_\DR^{-1}A)u_n]= \ab[(I-\la
A_\DR^{-1})u_n]\to0\,$ as $\,n\to\infty\,$. However, this is not
possible because $\,\la\not\in\si_\ess(A_\DR)\,$ and $\,u_n\,$ are
orthogonal to $\,E_\DR(\la)H\subset G_\la^0\,$. The obtained
contradiction proves the second statement of the lemma.

\subsection{Proof of Theorem~\ref{t1-main}}\label{Proofs4}
If $\,\la\not\in\si(A_\NR)\bigcup\si(A_\DR)\,$ then the theorem is
obvious because, in view of Lemmas \ref{l-g0} and
\ref{l1-projections}, we have
$\,n_\NR(\la)=n_\DR(\la)=n_{\NR,\DR}(\la)=0\,$ and
$\,(-\eps,\eps)\bigcap\si(\BC_\mu)=\varnothing\,$ for all
sufficiently small $\,\eps,\de>0\,$ and all
$\,\mu\in(\la-\eps,\la+\eps)\,$.

Suppose that $\,\la\,$ is an isolated eigenvalue. The first
statement of the theorem is an immediate consequence of
Lemma~\ref{l-g0}, so we only need to prove (2) and (3). Let us
choose $\,\eps\,$ and $\,\de\,$ as explained in Remark~\ref{r-main}
and assume, in addition, that $\,\de\,$ is so small that
$\,\la-\de>0\,$ and the interval $\,[\la-\de,\la+\de]\,$ does not
contain any points from $\,\si(A_\NR)\bigcup\si(A_\DR)\,$ with the
exception of $\,\la\,$.

Let $\,\LC_\mu\,$ be the subspace of $\,G_\mu\bigcap\HC\,$
spanned by the eigenfunction corresponding to the eigenvalues
$\,\nu_j(\mu)\,$ (see Remark~\ref{r-main}). By
Lemma~\ref{l-g0}, we have
$$
\LC_\la\ =\ E_\NR(\la)\HC\dotplus E_\DR(\la)\HC\ \subset\ G_\la^0\,.
$$
Therefore $\,E'_\BR(\la)\LC_\la\subset\LC_\la\,$, $\,\dim
E'_\NR(\la)\LC_\la=n_\DR(\la)\,$ and $\,\dim
E'_\DR(\la)\LC_\la=n_\NR(\la)\,$.

We are going to show that
\begin{multline}\label{p3}
\left|\,\ab\left[\BC_\mu
P'_\NR(\mu)u,\chi_{(-\eps,\eps)}(\BC_\mu)P'_\NR(\mu)u\right]-
(\mu-\la)\,\|u\|^2\,\right|\\
\leq\ C\,(\la-\mu)^2\,\|u\|^2\,,\qquad\forall u\in
E'_\NR(\la)\LC_\la
\end{multline}
and
\begin{multline}\label{p4}
\left|\,\ab\left[\BC_\mu
P'_\DR(\mu)u,\chi_{(-\eps,\eps)}(\BC_\mu)P'_\DR(\mu)u\right]-
(\la-\mu)\,\|u\|^2\,\right|\\
\leq\ C\,(\la-\mu)^2\,\|u\|^2\,,\qquad\forall u\in
E'_\DR(\la)\LC_\la\,,
\end{multline}
where $\,C\,$ is a constant independent of $\,u\,$ and
$\,\mu\in(\la-\eps,\la+\eps)\,$. From \eqref{p3} and \eqref{p4} it
follows that
\begin{align}
\label{p5} (\mu-\la)\,\ab[\BC_\mu w,w]\ &\geq\ 0\,,\qquad \forall
w\in\chi_{(-\eps,\eps)}(\BC_\mu)P'_\NR(\mu)E'_\NR(\la)\LC_\la\,,\\
\label{p6} (\la-\mu)\,\ab[\BC_\mu w,w]\ &\geq\ 0\,,\qquad \forall
w\in\chi_{(-\eps,\eps)}(\BC_\mu)P'_\DR(\mu)E'_\DR(\la)\LC_\la\,,
\end{align}
whenever $\,|\la-\mu|\,$ is small enough. If $\,\mu=\la\,$ then
$\,\chi_{(-\eps,\eps)}(\BC_\mu)P'_\BR(\mu)E'_\BR(\la)u=E'_\BR(\la)u\,$
for all $\,u\in\LC_\la\,$. By continuity, we have
$$
\dim\chi_{(-\eps,\eps)}(\BC_\mu)P'_\BR(\mu)E'_\BR(\la)\LC_\la\ =\
\dim E'_\BR(\la)\LC_\la\ =\ n_\NR(\la)+n_\DR(\la)-n_\BR(\la)
$$
for all $\,\mu\,$ sufficiently close to $\,\la\,$. Therefore the
estimates \eqref{p5} and \eqref{p6} imply the theorem (with another
positive $\,\de\,$).

In order to prove \eqref{p3} and \eqref{p4}, note that
$\,\bb[u]=0\,$ for all $\,u\in\LC_\la\,$ and, in view of \eqref{a6}
and \eqref{a7},
\begin{multline}\label{p7}
\ab\left[\BC_\mu P'_\NR(\mu)u,P'_\NR(\mu)u\right]\ =\
\bb[P'_\NR(\mu)u]\\ =\ (\mu-\la)\,\|u\|^2
+(\la-\mu)^2\,\bigl(u,R_\NR(\mu)u\bigr)\,,\qquad\forall u\in\LC_\la
\end{multline}
and
\begin{multline}\label{p8}
\ab\left[\BC_\mu P'_\DR(\mu)u,P'_\DR(\mu)u\right]\ =\
\bb[P'_\DR(\mu)u]\\ =\ (\la-\mu)\,\|u\|^2
-(\la-\mu)^2\,\bigl(u,R_\DR(\mu)u\bigr)\,,\qquad\forall
u\in\LC_\la\,.
\end{multline}
Therefore, for all $\,\mu\in(\la-\de,\la+\de)\,$ we have
\begin{multline}\label{p9}
\left|\,\ab\left[\BC_\mu
P'_\NR(\mu)u,P'_\NR(\mu)u\right]-(\mu-\la)\,\|u\|^2\,\right|\\ \leq\
C_1^{-1}\,(\la-\mu)^2\,\|u\|^2\,,\qquad\forall u\in
E'_\NR(\la)\LC_\la\,,
\end{multline}
and
\begin{multline}\label{p10}
\left|\,\ab\left[\BC_\mu
P'_\DR(\mu)u,P'_\DR(\mu)u\right]-(\la-\mu)\,\|u\|^2\,\right|\\ \leq\
C_1^{-1}\,(\la-\mu)^2\,\|u\|^2\,,\qquad\forall u\in
E'_\DR(\la)\LC_\la\,,
\end{multline}
where $\,C_1\,$ is the distance from $\,[\la-\de,\la+\de]\,$ to
$\,\left(\si(A_\NR\bigcup\si(A_\DR)\right)\setminus\{\la\}\,$.

Let $\,S_\BR\,$ be the projections onto $\,E_\BR(\la)\HC\,$ in
$\,\LC_\la\,$ such that $\,S_\NR E_\DR(\la)=0\,$ and $\,S_\DR
E_\NR(\la)=0\,$. Then $\,u=S_\NR u+S_\DR u\,$ for all
$\,u\in\LC_\la\,$. Since $\,\dim\LC_\la<\infty\,$ and
$\,E_\NR(\la)\HC\bigcap E_\DR(\la)\HC=\{0\}\,$, the projections
$\,S_\BR\,$ are well defined and bounded as operators from $\,H\,$
to $\,H^1\,$.

If $\,u\in\LC_\la\,$ then $\,P'_\BR(\mu)u=P'_\BR(\mu)(u-S_\BR
u)=(I-(\la-\mu)R_\BR(\mu))(u-S_\BR u)\,$ for all $\,\mu\ne\la\,$
and, by \eqref{b},
\begin{align*}
\ab\left[\BC_\mu P'_\NR(\mu)u,v\right]\ &=\
\ab\left[P'_\NR(\mu)u,v\right]-\mu\left(P'_\NR(\mu)u,v\right)
=(\mu-\la)\,(S_\DR u,v)\,,\qquad\forall v\in G_\mu\,,\\
\ab\left[\BC_\mu P'_\DR(\mu)u,v\right]\ &=\
\ab\left[P'_\DR(\mu)u,v\right]-\mu\left(P'_\DR(\mu)u,v\right)
=(\la-\mu)\,(S_\NR u,v)\,,\qquad\forall v\in G_\mu\,,
\end{align*}
for all $\,\mu\in(\la-\de,\la+\de)\,$. Since $\,(S_\NR
u,v)=\la^{-1}\,\ab[S_\NR u,v]\,$ and $\,(S_\DR
u,v)=\mu^{-1}\,\ab[S_\DR u,v]\,$ whenever $\,v\in G_\mu\,$, the
above identities imply that
\begin{align}
\label{p11} \ab\left[\BC_\mu P'_\NR(\mu)u,v\right]\ &=\
\mu^{-1}(\mu-\la)\,\ab\left[S_\DR u,v\right]\,,\qquad\forall
u\in\LC_\la\,,\ \forall v\in G_\mu\,,\\
\label{p12} \ab\left[\BC_\mu P'_\DR(\mu)u,v\right]\ &=\
\la^{-1}(\la-\mu)\,\ab\left[S_\NR u,v\right]\,,\qquad\forall
u\in\LC_\la\,,\ \forall v\in G_\mu\,.
\end{align}

In view of Lemma~\ref{l1-projections},
$\,\left.\left(I-\chi_{(-\eps,\eps)}(\BC_\mu)\right)P'_\BR(\mu)E'_\BR(\la)\right|_{\LC_\la}\,$
is an analytic operator-valued function of
$\,\mu\in(\la-\de,\la+\de)\,$. Since this operator-valued function
vanishes at $\,\mu=\la\,$, we have
\begin{multline}\label{p13}
\ab\left[\left(I-\chi_{(-\eps,\eps)}(\BC_\mu)\right)P'_\BR(\mu)E'_\BR(\la)u\right]\\
\leq\ C_2\,(\la-\mu)^2\,\ab[u]\ =\
C_2\,(\la-\mu)^2\,\la\,\|u\|^2\,,\qquad\forall u\in\LC_\la\,,
\end{multline}
with some positive constant $\,C_2\,$ independent of $\,\mu$ and
$u\,$. Substituting
$$
v\ =\ \left(I-\chi_{(-\eps,\eps)}(\BC_\mu)\right)P'_\BR(\mu)u
$$
into \eqref{p11}, \eqref{p12} and applying \eqref{p13}, we obtain
\begin{equation}\label{p14}
\ab\left[\BC_\mu
P'_\BR(\mu)u,\left(I-\chi_{(-\eps,\eps)}(\BC_\mu)\right)P'_\BR(\mu)u\right]
\leq C_3\,(\la-\mu)^2\,\|u\|^2\,,\quad\forall u\in
E'_\BR(\la)\LC_\la\,,
\end{equation}
with some constant $\,C_3\,$ independent of $\,\mu$ and $u\,$. Now
\eqref{p3} and \eqref{p4} follow from \eqref{p9}, \eqref{p10} and
\eqref{p14}.

\subsection{Proof of Corollary~\ref{c-main}}\label{Proofs5}

Lemma~\ref{l1-projections} and Theorem~\ref{t1-main} imply that the
function $\,\dim G_\la^-\,$ is constant on every connected component
of the set $\R\setminus(\si(A_\NR)\bigcup\si(A_\DR))$. If
$\,\la\not\in\si_\ess(A_\NR)\bigcup\si_\ess(A_\DR)\,$ and
$\,\la\in\La\,$ is an eigenvalue then, by Theorem~\ref{t1-main}
\begin{equation}\label{p15}
\dim G_\mu^-\ =\ \begin{cases}
\dim G_\la^-+n_\DR(\la)\,,&\forall \mu\in(\la-\de,\la),\\
\dim G_\la^-+n_\NR(\la)\,,&\forall \mu\in(\la,\la+\de)\,,
\end{cases}
\end{equation}
provided that $\,\de>0\,$ is small enough. In other words, the value
of $\,\dim G_\mu^-\,$ jumps by $\,n_\NR(\la)-n_\DR(\la)\,$ as
$\,\mu\,$ passes through the eigenvalue $\,\la\,$, and $\,\dim
G_\la^-=\dim G_{\la-0}^--n_\DR(\la)\,$. Summing up these jumps over
all the eigenvalues lying between $\,a\,$ and $\,b\,$, we obtain
\eqref{eq-main}.

\subsection{Proof of Theorem~\ref{t2-main}}\label{Proofs6}

Let $\,\LC\,$ be the subspace of $\,\chi_{(-\infty,\la)}(A_\NR)H\,$
spanned by all the vectors $v\in\chi_{(-\infty,\la)}(A_\NR)H$ such
that
\begin{equation}\label{p16}
\chi_{(-\infty,\la]}(A_\DR)\,(A-\la I)v\ =\ 0\,.
\end{equation}
The inclusion $\LC\subset\DC(A_\NR)$ implies that $\,\dim\LC\geq
N_\NR(\la)-N_\DR(\la)-n_\DR(\la)\,$ and $\bb[v]=0$ for all
$\,v\in\LC\,$. From the latter identity, \eqref{p16} and \eqref{a7}
it follows that $\,\bb[P'_\DR(\la)v]<\bb[v]\,$ for all nonzero
$\,v\in\LC\,$. Since $\,\left((A-\la I)v,v\right))<0\,$ for all
nonzero $\,v\in\chi_{(-\infty,\la)}(A_\NR)H\,$ and $\,\left((A-\la
I)v,v\right)\geq0\,$ for all $\,v\in\DC(A_\DR)\,$ satisfying
\eqref{p16}, we have $\,\LC\bigcap\DC(A_\DR)=\{0\}\,$. Therefore
$\,\ker\left.P'_\DR(\la)\right|_\LC\subset
L\bigcap\DC(A_\DR)=\{0\}\,$ and, consequently,
$$
\dim P'_\DR(\la)\LC\ \geq\ N_\NR(\la)-N_\DR(\la)-n_\DR(\la)\,.
$$
Thus we have $\,\dim G_\la^-\geq\dim P'_\DR(\la)\LC\geq
N_\NR(\la)-N_\DR(\la)-n_\DR(\la)\,$. Now the theorem follows from
\eqref{upper}.

\subsection{Proof of Corollary~\ref{c-resolvents}}\label{Proofs7}
We have $\,R'(\la)H\subset G_\la\,$ and, by \eqref{b},
\begin{equation}\label{p17}
\ab[\BC_\la R'(\la)u,v]\ =\ \ab[R'(\la)u,v]-\la(R'(\la)u,v)\ =\
(u,v) \ =\ \ab[A_\NR^{-1}u,v]
\end{equation}
for all $\,\forall u\in H\,$ and $\,\forall v\in G_\la\,$. The above
identity implies that $\,\BC_\la R'(\la)u=\Pi'_\la A_\NR^{-1}u\,$
for all $\,u\in H\,$. In view of Lemma~\ref{l-g0}, the operator
$\,\BC_\la\,$ is invertible and, consequently,
$\,R'(\la)=\BC_\la^{-1}\Pi'_\la A_\NR^{-1}\,$. Since
$\,\ab[A_\NR^{-1}u,v]=(u,v)$, the subspace $\,\Pi'_\la
A_\NR^{-1}H\,$ is $H^1$-dense in $\,G_\la\,$. Therefore
$\,R'(\la)H\,$ is an $\,H^1$-dense subspace of $\,G_\la\,$. Finally,
by \eqref{p17},
$$
\ab[\BC_\la R'(\la)u,R'(\la)v]\ =\ (u,R'(\la)v)\,,\qquad\forall
u,v\in H\,.
$$
Thus we have $\,\ab[\BC_\la u,u]<0\,$ on a $k$-dimensional subspace
of $\,G_\la\,$ if and only if $\,(R'(\la)u,u)<0\,$ on a
$k$-dimensional subspace of $\,H\,$. Now the corollary follows from
Theorem~\ref{t2-main}.

\subsection{Proof of Corollary~\ref{c-diff}}\label{Proofs8}
Let $\,a(\xi)\,$ be the full symbol of the operator $\,L^*L\,$, and
let $\,\la_1(\xi),\ldots\la_m(\xi)\,$ be the eigenvalues of
$\,a(\xi)\,$. Then $\,\la_{\DR,k}>\la_*:=\min_j\inf_\xi\la_j(\xi)\,$
for all $\,k\,$ because $\,\ab[u]\geq\la_*\,\|u\|^2\,$ on
$\,C_0^\infty(\Om)\,$. On the other hand, since $\,\la_j(\xi)\,$ are
continuous functions of $\,\xi\,$, the equation $\,\det(a(\xi)-\la
I)=0\,$ has infinitely many $\,\xi$-solutions for each fixed
$\,\la>\la_*\,$. Therefore $\,G_\la\,$ contains an infinite
dimensional set formed by functions of the form
$\,u_\xi=e^{ix\cdot\xi}\,\vec c\,$ where $\,\vec c\in\ker(a(\xi)-\la
I)\,$. For each of these functions we have
$\,\ab[u_\xi]=\la\,\|u_\xi\|^2\,$. This implies that either $\,\dim
G_\la^-\geq1\,$ or $\,\dim G_\la^0=\infty\,$. By Lemma~\ref{l-g0},
the latter is possible only if $\,\la\geq\la_{\NR,\infty}\,$.
Therefore, by Remark~\ref{r-g0}(2), we have
$\,\la_{\NR,k+1}\leq\la_{\DR,k}\,$ for all eigenvalues lying below
$\,\la_{\NR,\infty}\,$. If at least one Dirichlet eigenfunction
corresponding to $\,\la_{\DR,k}\,$ does not satisfy the Neumann
boundary condition then $\,n_\DR(\la_{\DR,k})\geq1\,$ and, by
Remark~\ref{r-g0}(1), $\,\la_{\NR,k+1}<\la_{\DR,k}\,$.

\section{Remarks}\label{R}

\subsection{$A_\DR$ and $A_\NR$ as self-adjoint extensions}\label{R1}
Denote $\,H_0^2:=H_0^1\bigcap\DC(A_\NR)\,$. Since
$\,\DC(A_\DR)=\Pi_0\DC(A_\NR)\,$ (see Subsection \ref{A1}), we have
$\,H_0^2=\DC(A_\DR)\bigcap\DC(A_\NR)\,$.

The $\,H$-adjoint $A^*$ coincides with the restriction of $A$ to
$H_0^2$. Indeed, if $\,(u,Av)=(\tilde u,v)\,$ for some $u,\tilde
u\in H$ and all $\,v\in H_A^1\,$ then, taking $v\in\DC(A_\NR)$ or
$v\in\DC(A_\DR)$, we obtain $\,u=A_\NR^{-1}\tilde u=A_\DR^{-1}\tilde
u\,$. Therefore $\,\DC(A^*)\subset H_0^2\,$. On the other hand, if
$\,u\in H_0^2\,$ then $\,(u,Av)=\ab[u,v]\,$ because $\,u\in H_0^1\,$
and $\,\ab[u,v]=(A_\NR u,v)\,$. Thus $\DC(A^*)=H_0^2$ and
$A^*=\left.A\right|_{H_0^2}$.

If $\,H_0^2\,$ is not dense in $\,H\,$ then the second adjoint
$A^{**}$ does not exist and the operator $\,A\,$ is not
closable in $\,H\,$ (see, for example, \cite[Section 3.3]{BS}).

If $\,H_0^2\,$ is dense in $\,H\,$ then $\,A_\DR\,$ and $\,A_\NR\,$
are self-adjoint extensions of $\,A^*\,$, and $\,A^{**}\,$ is the
closure of $\,A\,$. Note that $\,\DC(A)=H_A^1\,$ may be strictly
smaller than $\,H^1\bigcap\DC(A^{**})\,$. Also, the $\,H^1$-closed
subspaces $\,G_z\,$ may be strictly smaller than
$\,\left.\ker(A^{**}-zI)\right|_{H^1}\,$ and may not be closed in
$\,H\,$ (see the next subsection).

\subsection{An example}\label{R2}

Let $\,\Omega\,$ be a bounded domain with smooth boundary, and let
$\,H=L_2(\Om)\,$ and $\,H^s\,$ be the Sobolev spaces. If
$\,\ab[u]=\|\nabla u\|^2+\|u\|^2\,$ and $\,H_0^1\,$ is the
$\,H^1$-closure of $\,C_0^\infty(\Om)\,$ then $\,A=-\Delta+I\,$,
$\,\DC(A)=\{u\in H^1\,:\,Au\in H\}\,$ and $\,G_0=\{u\in
H^1\,:\,Au=0\}\,$. The self-adjoint operators $\,A_\DR\,$ and
$\,A_\NR\,$ are obtained by imposing the Dirichlet and Neumann
boundary conditions. The $\,H$-adjoint $\,A^*\,$ coincides with the
restriction of $\,A\,$ to $\,H_0^2:=\{u\in
H^2\,:\,\left.u\right|_{\partial\Om}=\left.\partial_n
u\right|_{\partial\Om}=0\}\,$, where $\,\partial_n\,$ is the normal
derivative. The second $\,H$-adjoint $\,A^{**}\,$ is the extension
of $\,A\,$ to $\,\DC(A^{**})=\{u\in H\,:\,Au\in H\}\,$, and
$\,\DC(A)=H^1\bigcap\DC(A^{**})\,$.

Let us choose a nonzero function $\,v_0\in G_0\,$, and define
$\,\tilde H_0^1=H_0^1\oplus\LC_0\,$, where $\,\LC_0\,$ is the one
dimensional subspace spanned by $\,v_0\,$ and $\,\oplus\,$ denotes
the orthogonal sum in $\,H^1\,$. Then the corresponding operator
$\,\tilde A\,$ is the same differential operator $\,-\Delta+I\,$ but
$\,\DC(\tilde A)=\{u\in\DC(A)\,:\,\langle
\partial_n P'_\NR(0)u,v_0\rangle_{\partial\Om}=0\}\,$, where
$\,\langle\cdot,\cdot\rangle_{\partial\Om}\,$ denotes the
sesquilinear pairing between $\,H^{-1/2}(\partial\Omega)\,$ and
$\,H^{1/2}(\partial\Omega)\,$. The Neumann operator remains the
same, and the domain of new ``Dirichlet'' operator is $\,\DC(\tilde
A_\DR)=\DC(\tilde A)\bigcap\tilde H_0^1\,$. Finally,
$$
\DC(\tilde A^*)\ =\ \DC(A_\NR)\bigcap\tilde H_0^1\ =\ \{u\in
H^2\bigcap\tilde H_0^1\,:\,\left.\partial_n
u\right|_{\partial\Om}=0\}
$$
and, consequently, $\,H_0^2\subset\DC(\tilde A^*)\,$.

Let $\,v_0\not\in H^2\,$. Then
$\,\left.v_0\right|_{\partial\Om}\not\in H^{3/2}(\partial\Om)\,$ and
$\,\left.u\right|_{\partial\Om}\not\in H^{3/2}(\partial\Om)\,$ for
all $\,u\in\tilde H_0^1\setminus H_0^1\,$. This implies that
$\,H^2\bigcap\tilde H_0^1=H^2\bigcap H_0^1\,$ and $\,\DC(\tilde
A^*)=\DC(A^*)=H_0^2\,$. Thus we have $\,\tilde A^{**}=A^{**}\,$. By
the above, in this case $\,\DC(\tilde A)\ne H^1\bigcap\DC(\tilde
A^{**})\,$.

The $\,H^1$-orthogonal complement $\,\tilde G_0:=H^1\ominus\tilde
H_0^1=G_0\ominus\LC_0\,$ coincides with the kernel of the functional
$\,u\to\ab[v_0,u]=\langle\partial_nv_0,u\rangle_{\partial\Om}\,$
defined on the space $\,G_0\,$. If $\,v_0\not\in H^2\,$ then this
functional is not $\,H$-continuous and $\,\tilde G_0\,$ is not
$\,H$-closed. Now \eqref{a1} implies that $\,\tilde G_z:=\ker(\tilde
A-zI)\,$ are not $H$-closed for all $\,z\in\C\,$.

\subsection{The projections $P_\BR(\la)$}\label{R3}

Note that, by the spectral theorem, the right hand side of
\eqref{p7} is a nondecreasing function of $\,\mu\,$ and the right
hand side of \eqref{p8} is a nonincreasing function of $\,\mu\,$.
This observation allows one to simplify the proof of
Theorem~\ref{t1-main} in the case where $\,\dim G_\la^-<\infty\,$ or
$\,\dim G_\la^+<\infty\,$. The monotonicity is an implicit
consequence of the following result.

\pagebreak

\begin{lemma}\label{l2-projections}
Let $\La$ be an arbitrary real interval, and let $v\in H_A^1$.
\begin{enumerate}
\item[(1)]
If $\,\La\bigcap\si_\ess(A_\NR)=\varnothing\,$ and
$\,\chi_\La(A_\NR)P'_\NR(0)v=0\,$ then $\,P'_\NR(\la)v\in G_\la\,$
for all $\la\in\La\,$ and $\,\bb[P'_\NR(\la)v]\,$ is a
non\-decreasing function on $\La$.
\item[(2)]
If $\,\La\bigcap\si_\ess(A_\DR)=\varnothing\,$ and
$\,\chi_\La(A_\DR)P'_\DR(0)v=0\,$ then $\,P'_\DR(\la)v\in G_\la\,$
for all $\la\in\La\,$ and $\,\bb[P'_\DR(\la)v]\,$ is a nonincreasing
function on $\La$.
\end{enumerate}
If, in addition, $v\not\in\DC(A_\BR)$ then the function
$\,\bb[P'_\BR(\la)v]\,$ is strictly monotone.
\end{lemma}

\begin{proof}
Since $\,(A-\la I)P'_\BR(\la)v\ =\ E_\BR(\la)(A-\la I)v\,$, the
equality $\,\chi_\La(A_\BR)P'_\BR(0)v=0\,$ implies that $\,(A-\la
I)P'_\BR(\la)v=-\la E_\BR(\la)P'_\BR(0)v=0\,$. Thus we have
$\,P'_\BR(\la)v\in G_\la\,$ for all $\,\la\in\La\,$.

If $\,w_\BR\in\DC(A_\BR)\,$ then
$\,P'_\BR(\la)w_\BR=E_\BR(\la)w_\BR\,$. Using this identity, one can
easily show that
\begin{eqnarray}
\label{r2}
\bb[P'_\NR(\la)(w+w_\NR)]&=&\bb[P'_\NR(\la)w]-(E_\NR(\la)(A-\la I)w,w_\NR)\,,\\
\label{r3}
\bb[P'_\DR(\la)(w+w_\DR)]&=&\bb[P'_\DR(\la)w]+(w_\DR,E_\DR(\la)(A-\la
I)w)
\end{eqnarray}
for all $\,w\in H_A^1\,$, $w_\NR\in\DC(A_\NR)$ and
$w_\DR\in\DC(A_\DR)\,$. Since $E_\BR(\la)P'_\BR(0)v=0$ and
$\,P'_\BR(\la)P'_\BR(0)=P'_\BR(\la)\,$ for all $\,\la\in\La\,$,
substituting $w=P'_\BR(0)v$, $w_\BR=P_\BR(0)v$ in \eqref{r2},
\eqref{r3} and applying \eqref{a6}, \eqref{a7}, we obtain
\begin{eqnarray*}
&\bb[P'_\NR(\la)v]=\bb[P'_\NR(0)v]+\la\|P'_\NR(0)v\|^2
+\la^2(R_\NR(\la)P'_\NR(0)v,P'_\NR(0)v)\,,&\quad\forall\la\in\La,\\
&\bb[P'_\DR(\la)v]=\bb[P'_\DR(0)v]-\la\|P'_\DR(0)v\|^2
-\la^2(R_\DR(\la)P'_\DR(0)v,P'_\DR(0)v)\,,&\quad\forall\la\in\La.
\end{eqnarray*}
Now the required monotonicity results follow from the spectral
theorem.
\end{proof}

Note that $\,P'_\BR(0)=P'_\BR(0)P'_\BR(\mu)=P'_\BR(\mu)-\mu
A_\BR^{-1}P'_\BR(\mu)\,$ whenever $\,\mu\not\in\si(A_\BR)\,$.
Therefore we have $\,\chi_\La(A_\BR)P'_\BR(0)v=0\,$ if and only if
$\,\chi_\La(A_\BR)P'_\BR(\mu)v=0\,$ for all
$\mu\not\in\si(A_\BR)\,$.

\subsection{Analytic properties of $\Pi(z)$}\label{R4}

If the embedding $H_0^1\hookrightarrow H$ is compact then the
operator-valued functions $\,\Pi(z)\,$ and $\,\Pi'(z)\,$ introduced
in Subsection~\ref{A3} are meromorphic in the whole complex plane.
Indeed, since $\,\ab[A_\DR^{-1}u]=(A_\DR^{-1}u,u)\,$, the
compactness of the embedding $H_0^1\hookrightarrow H$ implies that
$A_\DR^{-1}$ is compact as an operator from $H$ to $H^1$.
Consequently,
$$
T_z^\star\,T_z-I\ =\
\left.z^2\,\Pi_0A_\NR^{-2}\right|_{H_0^1}-\left.2z\,\Pi_0A_\NR^{-1}\right|_{H_0^1}\
=\
\left.z^2\,A_\DR^{-1}A_\NR^{-1}\right|_{H_0^1}-\left.2z\,A_\DR^{-1}\right|_{H_0^1}
$$
are compact operators in $H_0^1$. Now the required result follows
from the analytic Fredholm theorem (see, for example, \cite[Section
1.8]{Ya}).

\end{document}